\documentclass{amsart}
\usepackage{amssymb,amsmath,amsthm,amscd}
\newtheorem{theorem}{Theorem}

\newtheorem{proposition}{Proposition}

\newtheorem{corollary}{Corollary}

\begin{document}

\title{The Cyclic Vector Lemma}
\author{Andy~R. Magid}
\address{Department of Mathematics\\
        University of Oklahoma\\
        Norman OK 73019\\
 }
 \email{amagid@ou.edu}
\subjclass{12H05}
\date{December 8, 2022}
\maketitle

\begin{abstract}
Let $F$ be a differential field of characteristic zero with algebraically closed constant field $C$. Let $E$ be a Picard--Vessiot closure of $F$, $R \subset E$ its Picard--Vessiot ring and $\Pi$ the differential Galois group of $E$ over $F$. Let $V$ be a differential $F$ module, finite dimensional as an $F$ vector space. Then $V$ is singly generated as a differential $F$ module if and only if there is a $\Pi$ module injection $\text{Hom}_F^\text{diff}(V,R) \to R$. If $C \neq F$ such an injection always exists.
\end{abstract}

Let $F$ be a differential field of characteristic zero with algebraically closed constant field $C$. Denote the derivation of $F$ by $D_F$, or just $D$ if the context is clear. A differential $F$ module $V$ is an $F$ vector space with an additive endomorphism $D_V$, or just $D$ is the context is clear, satisfying $D_V(av)=D_F(a)v+ aD_V(v)$ for all $a \in F$ and $v \in V$. If there is a $v \in V$ such that $V=\{D^nv+a_{n-1}D^{n-1}v+ \dots +a_0v \vert \text{ for all } n \geq 0 \text{ and for all } a_0, \dots, a_{n-1} \in F\}$ then $V$ is said to be singly generated, and $v$ is said to be a cyclic vector for $V$. It is a theorem that if $V$ is finite dimensional and $F \neq C$ then cyclic vectors aways exist.  This seems to be originally due to Cope \cite[p. 132]{c} for systems of linear differential equations with coefficients in $F=\mathbb C(x)$; a modern proof for modules over differential rings was given by Katz \cite[Thm. 1., p. 65]{k} (reproduced in \cite[Prop. 2.9, p. 40]{vps}). The proof of Katz is brief and elegant. Nonetheless, it may be useful to provide another, perhaps more conceptual, argument, which this paper does, that can give some inclination of why cyclic vectors should exist when $F \neq C$, and what goes wrong when $F=C$.

A differential homomorphism between differential $F$ modules $V$ and $W$ is an $F$ linear homomorphism $f:V \to W$ satisfying $fD_V=D_Wf$. With this definition differential modules form an abelian category. 

The set of all differential homomorphisms  $V \to W$ is denoted $\text{Hom}_F^\text{diff}(V,W)$. Note that this latter is a $C$ vector space.
Let $E$ be a Picard--Vessiot closure of $F$, $R \subset E$ its Picard--Vessiot ring and $\Pi$ the differential Galois group of $E$ over $F$. Since $R$ is a differential $F$ module, for any differential $F$ module $V$ we can form $\text{Hom}_F^\text{diff}(V,R)$. The group $\Pi$ acts on $R$ and hence on $\text{Hom}_F^\text{diff}(V,R)$ making the latter a $\Pi$ module. Thus we can consider $\text{Hom}_\Pi( \text{Hom}_F^\text{diff}(V,R),R)$. As in all such ``double-dual" situations, we then have a map $V \to \text{Hom}_\Pi\text{Hom}_F^\text{diff}(V,R)$ by $v \mapsto \hat{v}$ where for $f \in  \text{Hom}_F^\text{diff}(V,R)$ we define $\hat{v}(f)=f(v)$. (We omit the calculation that $\hat{v}$ is $\Pi$ equivariant.). Then by \cite{m}

\begin{proposition}\label{tannaka}
If $V$ is a differential $F$ module finite dimensional of dimension $n$ as an $F$ vector space, then 
\begin{enumerate}
\item[(a)] $\text{Hom}_F^\text{diff}(V,R)$ is a rational $\Pi$ module of dimension $n$ over $C$\\
\item[(b)]  $V \to \text{Hom}_\Pi( \text{Hom}_F^\text{diff}(V,R),R)$ by $v \mapsto \hat{v}$ is an isomorphism of differential $F$ modules\\
\end{enumerate}
\end{proposition}

\begin{proof}
\begin{enumerate}
\item[(a)] \cite[Prop. 3.6, p382]{m}\\
\item[(b)] \cite[Thm. 3.7 1., p382]{m}\\
\end{enumerate}
\end{proof}

These results imply a relative cyclic vector theorem:

\begin{theorem} \label{relativecyclicvector} Let $V$ be a finite dimensional differential $F$ module. Suppose that there is a $\Pi$ module injection $\phi: \text{Hom}_F^\text{diff}(V,R) \to R$. Let $v \in V$ such that $\hat{v}=\phi$ then $v$ is a cyclic vector for $V$. Conversely, let $v \in V$ be a cyclic vector. Then $\hat{v}: \text{Hom}_F^\text{diff}(V,R) \to R$ is a $\Pi$ module injection.
\end{theorem}

\begin{proof} Since $\phi$ is injective, if $f \in \text{Hom}_F^\text{diff}(V,R)$ and $\phi(f)=0$ then $f=0$. Since $\phi=\hat{v}$ and $\hat{v}(f)=f(v)$, this says that $f(v)=0$ implies $f=0$. Now let $V_0 \subseteq V$ be $V_0=\{D^nv+a_{n-1}D^{n-1}v+ \dots +a_0v \vert \text{ for all } n \geq 0 \text{ and for all } a_0, \dots, a_{n-1} \in F\}$. Suppose $V_0 \subsetneq V$ and let $W=V/V_0$. The surjection $V \to W$ is, of course, differential. Since $W$ has positive $F$ dimension, by Proposition \ref{tannaka} (a)  $\text{Hom}_F^\text{diff}(W,R)$ has positive $C$ dimension, and is in particular not zero. Let $g: W \to R$ be a non-zero differential morphism. Then $f=gp:V \to R$ is a non-zero differential homomorphism with $f(v)=0$. This contradiction shows that $V=V_0$, namely that $v$ is a cyclic vector for $V$.

Now suppose $v \in V$ is a cyclic vector. This means that $v$ generates $V$ as a differential $F$ module. Therefore if $f: V \to R$ is a differential $F$ homomorphism, $f=0$ if and only if $f(v)=0$. Thus $\hat{v}: \text{Hom}_F^\text{diff}(V,R) \to R$ is an injection.
\end{proof}

Theorem \ref{relativecyclicvector} shows that finding cyclic vectors is equivalent to finding injections from the $C$ finite dimensional rational $\Pi$ module $\text{Hom}_F^\text{diff}(V,R)$ to $R$. We consider this possibility for any $C$ finite dimensional rational $\Pi$ module:

\begin{theorem} \label{embedmodule}
Let $W$ be a $C$ finite dimensional rational $\Pi$ module. 
\begin{enumerate}
\item[(a)] If $C \subsetneq F$ then there exists a $\Pi$ module injection $W \to R$.\\
\item[(b)] If $F=C$ then there is a $\Pi$ module injection $W \to R$ if and only if $W$ is isomorphic to a $\Pi$ submodule of $C[\Pi]$.\\
\end{enumerate}
\end{theorem}

\begin{proof} The action of $\Pi$ on $W$ induces a homomorphism $\Pi \to \text{GL}(W)$ whose kernel $H$ is a (normal) proalgebraic and whose quotient $G=\Pi/H$ is an affine algebraic group. Under the action of $\Pi$ on $W,$ $H$ acts trivially, so $W$ is a (finite dimensional) $G$ module.  The action of $G$ on $W$ defines a map $\nabla_W:W \to W_t \otimes C[G]$ by $w \mapsto \sum w_i \otimes h_i$ where for all $g \in G$ we have $g\cdot w= \sum h_i(g)w_i$. Here $W_t$ means the vector space $W$ with trivial $G$ action. One verifies that $\nabla_W$ is $G$ linear and injective. Let $K=E^H$. Then $K$ is a Picard--Vessiot extension of $F$ with differential Galois group $G$. The Picard--Vessiot ring $S$ of $K$ is $R^H$, however all we need to know is that $S \subseteq R$. By \cite[Prop. 2.2, p. 235]{cjm}, $S$ is isomorphic to $F \otimes_C C[G]$ as an $F$ vector space and $G$ module. As a $G$ module $F$ has trivial action. Since $C$ is algebraically closed, if $C \subsetneq F$ then $F$ is infinite dimensional over $C$ and hence there is a ($G$ module) injection $W_t \to F$; thus $\nabla_W$ induces a $G$ module injection $W \to W_t \otimes C[G] \to F \otimes C[G]$ and the latter is isomorphic as an $F$ vector space and $G$ module to $S \subset R$. This proves assertion (a). 

If $W^H=W$,  any $\Pi$ module homomorphism, injective or not, from $W$ to $R$ has image in $R^H$, and the latter is isomorphic to $F \otimes_C C[G]$ as an $F$ vector space and $G$ module. When $F=C$, $R^H$ is then isomorphic to $C[G]=C[\Pi]^H$. Under these identifications, (b) is a tautology.

Combining Theorems \ref{relativecyclicvector} and \ref{embedmodule} we have the cyclic vector theorem:

\begin{corollary} \label{cyclicvectortheorem}
Assume $F \neq C$. Then every finite dimensional differential $F$ module has a cyclic vector.
\end{corollary}

\end{proof}

\end{document}